\documentclass[10pt]{article}
\usepackage[utf8]{inputenc}
\usepackage{commath} 
\usepackage{pdfsync} 
\usepackage[colorlinks]{hyperref}

\usepackage{amsmath} 
\usepackage{amssymb}
\usepackage{amsfonts,amssymb}
\usepackage{amsthm}
\usepackage{enumerate,xspace,threeparttable}
\usepackage{graphicx}

\usepackage{tikz}
\usepackage{verbatim}
\usepackage{animate}
\usetikzlibrary{arrows,shapes}
\usetikzlibrary{fit}
\usetikzlibrary{calc,through,backgrounds}
\usetikzlibrary{shapes,snakes,positioning}
\usetikzlibrary{decorations.shapes}
\usetikzlibrary{decorations.markings}

\pgfdeclarelayer{background}
\pgfdeclarelayer{bbackground}
\pgfsetlayers{bbackground,background,main}

\tikzstyle{vertex}=[circle,fill=black!25,minimum size=15pt,inner sep=1pt]
\tikzstyle{source}=[circle,fill=red!25,minimum size=20pt,inner sep=1pt]
\tikzstyle{sink}=[circle,fill=green!25,minimum size=20pt,inner sep=1pt]
\tikzstyle{selected vertex} = [vertex, fill=red!24]
\tikzstyle{edge} = [draw,thick,-latex]
\tikzstyle{weight} = [font=\tiny]
\tikzstyle{selected edge} = [draw,line width=5pt,-,red!50]
\tikzstyle{ignored edge} = [draw,line width=5pt,-,black!20]
\tikzstyle{diedge} = [draw,thick,-latex]
\tikzstyle{rw} = [bend left=10,black!20,->]

\newcommand{\myarrow}[1]{%
\ifthenelse{\lengthtest{#1 pt < 0pt}}%
{\arrowreversed[red]{stealth}}%
{\arrow[blue]{stealth}}
}
\tikzstyle{fl} = [%
draw,thick,
decorate,
decoration={markings,mark=at position .5 with {\myarrow#1;}; }
]

\tikzstyle{bfl} = [fl=#1,red]
\tikzstyle{ofl} = [fl=#1,green]
\tikzstyle{flowval} = [font=\tiny,blue,auto=left]
\tikzstyle{psival} = [font=\tiny,blue,swap]

\tikzstyle{flv} = [font=\small]
\tikzstyle{blv} = [flv]
\tikzstyle{flvl} = [font=\small,swap]

\newtheorem{thm}{Theorem}

\newtheorem{prop}{Proposition}
\newtheorem{lem}{Lemma}

\theoremstyle{remark}

\let\qr=\eqref

\renewcommand{\d}{d\,}
\newcommand{\dd}[2]{\dfrac{d#1}{d#2}}
\newcommand{\RR}{\mathbb{R}}

\let\mc=\mathcal

\newcommand{\ab}[1]{\left|#1\right|}
\newcommand{\Ordo}[1]{{O(#1)}} 
\newcommand{\ordo}[1]{{o(#1)}} 
\newcommand{\OrdoOmega}[1]{{\varOmega(#1)}} 
\newcommand{\OrdoTheta}[1]{{\Theta(#1)}}

\newcommand{\aOrdoOmega}[1]{\ab{\OrdoOmega{#1}}}

\newcommand{\ddt}[1][t]{{\dd{}{#1}}}

\newcommand{\G}{\mathsf{G}}
\newcommand{\cy}{\sigma}
\newcommand{\lcy}{\xi}
\newcommand{\g}{\gamma}

\newcommand{\T}{^{T}}
\renewcommand{\t}[1][t]{({#1})}

\renewcommand{\ij}[1][ij]{_{#1}}
\newcommand{\kl}[1][kl]{\ij[#1]}
\renewcommand{\l}{\ell}
\newcommand{\J}{B}
\newcommand{\f}{\phi}
\newcommand{\ff}{f}
\newcommand{\ka}{b^{*}}

\newcommand{\p}{p}
\newcommand{\q}{q}
\let\e=\epsilon
\let\tl=\tilde

\newcommand{\supp}{\operatorname{supp}}

\newcommand{\lv}[1]{\l\left(#1\right)}
\newcommand{\ls}[2]{\lv{#1#2}}

\newcommand{\xpart}[1]{^{\mathbf #1}}
\newcommand{\ppart}{\xpart+}
\newcommand{\npart}{\xpart-}
\newcommand{\minp}{\operatorname{min_+}}
\newcommand{\xsupp}[2]{{#2}^{\mathbf #1}}
\newcommand{\gp}{{\xsupp+\gamma}}
\newcommand{\gm}{{\xsupp-\gamma}}

\newcommand{\diam}{\operatorname{diam}}
\newcommand{\diag}{\operatorname{diag}}

\newcommand{\BFS}{\mathcal{T}} 
\newcommand{\bfs}{\tau} 
\newcommand{\BOS}{\hat{\BFS}}
\newcommand{\ip}{\p^*} 
\newcommand{\iH}{H^*}
\newcommand{\ipsi}{\psi^*}

\newcommand{\ett}{{\mathbf1}}
\newcommand{\diett}{{\boldsymbol{\xi}}}
\newcommand{\rst}[1]{_{\vert{#1}}}
\newcommand{\Rst}[1]{\vert_{#1}}
\renewcommand{\norm}[1]{\lVert#1\rVert}
\newcommand{\normi}[1]{\norm{#1}_\infty}

\newcommand{\tbfs}{bfs\xspace}
\newcommand{\HH}{\Hat{H}}

\title{Convergence Properties for the Physarum Solver}
\author{%
\\
{Kentaro Ito}\\ {Dept.\ of Mathematical and Life Sciences, Hiroshima University}
\and
  {Anders Johansson}\\ {Dept.\ of Mathematics, Uppsala University}
\and
  {Toshiyuki Nakagaki}\\ {Dept.\ of Information Sciences, Future University Hakodate}
\and
  {Atsushi Tero}\\ {Research Institute for Electronic Sci., Hokkaido
  University}
}

\begin{document}
\maketitle

\begin{abstract}
  The Physarum solver \cite{teroetalphysA,onishi} is an intuitive
  mechanism for solving optimisation problems. The solver is based on
  the idea of a current reinforced electrical network, whereby the
  conductivity $\cy\t\in\RR_+^E$ is reinforced by the current or flow,
  $\f\t\in\RR_+^E$.  In this paper, we show how the Physarum solver
  obtains the solution to the linear transshipment problem on a
  digraph $G=(V,E)$.  We prove that the current $\f\t$ and $\cy\t$
  converge with an exponential rate to a positive flow minimising
  $\l(\f)=\sum_{ij}\l\ij\f\ij$.  The limit flow has full support on
  the optimal set $\HH$.  If we assume that $\HH$ is a connected
  subgraph, the electrical potential vector $\p\t\in\RR^V$ converges
  to a solution $\p^*$ of the dual problem which is a discrete
  \emph{$\infty$-harmonic function} (\cite{richmangames,tug-of-war09})
  defined on the vertices of a subgraph $\iH$ which in many cases is a
  spanning subgraph, i.e.\ $V(\iH)=V(G)$.
\end{abstract}

\section{Introduction}

Many biological systems solve problems in a decentralized manner.  As
a leading example, path-finding by {\it Physarum polycephalum}\/ (a
giant amoeba of true slime mold) is well studied.  Physarum can find
the shortest path in a maze and the risk-minimum path in the
inhomogeneous field of risk \cite{Fermat,Maze1,Maze2}. On the basis of
experimental work, a differential equation model of such path
minimisation has been proposed \cite{toshietalJTB,SoftMatter}. The
model is based on the idea of a current reinforced electrical network.
As current or flow increases through a medium the conductivity is
increased.

The Physarum solver has been primarily studied by numerical
simulation.  Indeed, the algorithm was extended in order to make it
applicable to a wider range of problems concerning the optimal design
of networks, e.g. the shortest network problem \cite{PASP} and
multi-objective optimisation \cite{TheoBiosci,teroetalscience}.
Numerical simulation shows a good performance in problem solving and a
correspondence to some interesting characteristics of the biological
system \cite{toshinewgen}. Onishi et al.\ have published a few papers
on the rigorous analysis of the Physarum solver
\cite{proof1,onishi,proof3}.  The work is pioneering and just at the
beginning.

A mathematical description of the Physarum solver is as follows. We
consider a weighted directed graph $G=(V,E,\l)$, where the positive
weights $\l\ij>0$ are interpreted as lengths of the arcs $ij\in E$.
For the physarum solver the candidate flow $\f\t$, $t>0$, is obtained
as an electric current in an electrical resistor network obtained from
an adapting positive conductivity vector $\cy\t = (\cy\ij(t))$,
$t\geq0$.  We apply fixed external current sources, described by a
fixed source vector $b=(b_i)$ which defines the flow requirements, and
we obtain from $\cy\t$ the current $\f\t=(\f\ij\t)$ and potential
$\p\t=(\p_i\t)$ using Kirchhoff's laws.  The physarum solver is then
completely specified by updating the conductivities according to the
recipe
\begin{equation}\label{e.dicondxx}
  \ddt\cy_{ij}\t + \cy\ij\t = \f\ij\t. 
\end{equation}
We can arbitrarily choose the initial state $\cy\ij\t[0]>0$, for all
$ij\in E$.  We show that $\f\t$ converges at an exponential rate to a
\emph{positive} flow $\hat\f$ that minimises the \emph{cost}
$\l(\f):=\sum\ij\l\ij\f\ij$.  In \cite{dennis59} and \cite{minty60}
minimum cost flow problems are solved using electrical networks with
\emph{non-linear} resistors. In contrast, the Physarum Solver works by
modifying networks with ordinary linear resistors.

In this report, we state a couple of theorems and proofs regarding the
algorithm by means of mathematical analysis. Using the above model, we
obtain a more general result than that obtained in \cite{onishi} which
gives the convergence for certain shortest-path problems in planar
graphs.  As stated above, we show convergence to a solution of the
\emph{transshipment problem} \cite{tardossurv} for the given weighted
digraph.  One consequence is that the shortest path in a maze is
certainly obtained and the analysis indicates additional utility and
performance of the algorithm. 

We also show that $\hat\p=\lim\p\t$ converge to a canonical dual
solution in the form of an \emph{$\infty$-harmonic function}, which,
generally, is defined outside the node support of the minimum flows.
Discrete infinity-harmonic functions have been introduced in
\cite{richmangames,tug-of-war09} as value functions for certain type
of games.  (Our definition of discrete $\infty$-harmonic functions is
slightly different for directed problems and take into consideration
the direction of the arcs in another way.)  The discrete
infinity-laplacian is highly non-linear, but here we obtain the
harmonic solution as a limit using solutions of the ordinary linear
laplacians $L(\cy\t)$.

What perhaps makes the Physarum solver stand out as an algorithm is
that its implementation is physically immediate; the computation
relies on physical quantities like conductivity, potential and flow,
which are present for many natural systems; in particular, they can be
derived from underlying diffusion processes.  This should make the
Physarum solver a candidate for a general description of how optimal
transport is handled in nature.  It should also be considered as a
possible basis for decentralised computer algorithms for mathematical
programming. We discuss briefly these issues in Section \ref{s.discuss}.

\subsection{Preliminaries}\label{s.prelim}

A vector $\f\in\RR^A$ is a real-valued functions on the finite
index-set $A$. We write $x=(x_\alpha)$, $\alpha\in A$.  The support
$\supp x$ is the set of $\alpha$ where $x_\alpha\not=0$. Arithmetic
operations and relations between vectors should be interpreted
component-wise, for example $xy^2/\l\leq z$ should mean the vector
$(x_\alpha y^2_\alpha/\l_\alpha \leq z_\alpha)$. Scalars are
interpreted as constant vectors of appropriate dimensions when
needed. We restrict the domain of vector $x$ to $B\subset A$ by
writing $x\rst{B}$ or $x\Rst B$.  The characteristic function for a
set $S\subset A$ is denoted $\ett_S$.  When using matrix-algebra, all
vectors are column vectors and the corresponding row-vector is denoted
by $x\T$. For a vector $x=(x_\alpha)$, we use $|x|$ to denote the
vector $(|x_\alpha|)$. We use $x\ppart$ and $x\npart$ to denote the positive
and negative parts, respectively, so that $x=x\ppart-x\npart$ and
$|x|=x\ppart+x\npart$. The norm $\|x\|$ is used to denote the $\ell_1$-norm of
$x$, i.e.\ $\|x\|=\sum_\alpha |x_\alpha|$, and the norm $\|x\|_\infty$
is the maximum norm $\|x\|_\infty = \max_\alpha |x_\alpha|$.

The language from graph theory is hopefully standard. However, we
consider mainly directed graphs $G=(V,E)$ without loops and multiple
edges. The term \emph{graph} will usually mean such digraphs.
Elements of $V$ are called \emph{nodes} (or \emph{vertices}) and are
denoted $i,j,k,l$ etc.; the elements of $E\subset V\times V$ are
called directed \emph{arcs} (\emph{edges}) and are denoted $ij$ $kl$,
etc.  For the arc $ij$, $i$ is the tail and $j$ is the head of the
arc.  We sometimes write $|G|$ for the number of edges $|E(G)|$.  The
\emph{underlying undirected graph} of a digraph $G$ is denoted by
$U(G)$.  An \emph{oriented subgraph} $H$ of $G$ is a digraph $H$ such
that $U(H)$ is a subgraph of $U(G)$. An oriented graph corresponds
with a vector $\diett_H\in\{0,1,-1\}^E$, where $\diett_H(ij)=1$ if $ij\in
H$, $-1$ if $ji\in H$ and zero otherwise. The oriented subgraph $H$ is
thus a \emph{directed subgraph} if $\diett_H=\ett_{E(H)}$. As
oriented subgraphs goes, we will mainly deal with oriented paths,
oriented cycles and oriented cuts (cut-sets) of $G$, and if the
corresponding subgraph is a sub-digraph of $G$ we talk of a
\emph{directed} cycle, path or cut in $G$.

We use standard asymptotic notation. Thus $g=\OrdoOmega f$ means
$\liminf |g|/|f| > 0$, $g=\Ordo f$ that $\limsup |g|/|f| <\infty$ and
$g=\ordo f$ that $\limsup |g|/|f|=0$, as the relevant limit is
taken. 

\subsection{The transshipment problem}

In a discrete setting the transshipment problem has the formulation of
a \emph{minimum cost flow problem} on a directed \emph{connected}
graph $G=(V,E)$ without upper capacities and positive linear costs.  A
\emph{flow} (or a \emph{$b$-flow}) is an arc-vector
$\f=(\f_{ij})_{ij\in E} \in \RR^{E}$ that satisfies, for all nodes
$i\in V$, Kirchhoff's Current Law
\begin{equation}
  \label{e.kirchoff1}
  \sum_{ij\in E} \f_{ij} - \sum_{ki \in E} \f_{ki} = b_i.
\end{equation}
The \emph{source vector} is the node vector $b\in\RR^V$ obtained as
the prescribed net out-flow at nodes. We assume that $\sum_i b_i = 0$.
If $b_i>0$ we say $i$ is a supply node (or source) and if $b_i<0$ a
demand node (or sink).  A flow is positive if $\f\ij\geq0$, meaning
that the flow goes in the direction of the arcs.
The cost is a positive linear function on flows
$$ \l(\f) := \l\T\f = \sum_{ij\in E} \l\ij\f\ij,\qquad \l\ij\geq 0. $$
The cost coefficients $\l\ij>0$ will be referred to as \emph{lengths}
prescribed to the arcs.  

The transshipment problem is to minimise the cost among all
\emph{positive flows}.  It can be shown (see
e.g.~\cite{tardossurv,wagneruncap}) to be equivalent to the more
general minimum cost flow problem and subsumes among others the
shortest path problem and the minimum cost assignment problem.  An
\emph{undirected problem} is a problem without the positive flow
constraint, with cost function $\l(|\f|)$. It can be mapped to an
equivalent transshipment problem if we replace each arc with a pair of
arcs in both directions.

For the given directed graph $G=(V,E)$, let
$\J=\J_G\in\RR^{V\times E}$ denote the \emph{boundary operator} to
$G$, i.e.\ the matrix given by
$$
\J_{ie}:=
\begin{cases}
  -1&\text{if $i$ is the head of arc $e$}\\
  +1&\text{if $i$ is the tail of arc $e$}\\
  0&\text{if $i$ is not incident with $e$.}
\end{cases}
$$
The linear transshipment problem can be stated quite effectively using
the matrix $\J$, i.e.\ 
\begin{equation}
  \text{minimise $z=\l\T\f$, where $\J \f = b$ and $\f\geq 0$, }
  \label{e.primal}
\end{equation}

For the given source vector $b\in\RR^V$, we denote the affine space of
corresponding flows by $\Phi=\Phi(b,G) \subset \RR^G$. The convex
\emph{polyhedron} of positive flows is denoted by
$\Phi^+=\Phi^+(b,G)$, i.e. $\Phi^+ = \{\f\in\RR^G: \J\f = b, \f\geq
0\}$. The target for the Physarum Solver is to reach the set of flows
$\f\in\Phi^+$ where the cost $\l(\f)$ is minimum. Since $\l>0$, this
makes up a convex \emph{polytope} (i.e.\ a bounded polyhedron)
$\Hat\Phi=\Hat\Phi(b,G,\l)$.

A \emph{cut} is a partition $(S,V\setminus S)$ of $V$ in two parts and
the corresponding \emph{oriented cut} is given by the vector
$\kappa^S:=\J\T \ett_S \in\{-1,+1,0\}^E$.  Given a flow
$\f\in\Phi(b,G)$, it holds that for any cut $\kappa = \kappa^{S}$ the
\emph{net flow} $\kappa\T\f$ equals
\begin{equation}\label{e.kadef}
  b(S) := b\T \ett_S = \sum_{i\in S} b_i = -\sum_{i\in V\setminus S} b_i.
\end{equation}
The transshipment problem \eqref{e.primal} is \emph{feasible}, i.e.\
$\Phi^+\not=\emptyset$, if and only if, for every $S\subset V$ with
$b(S)>0$, there is some arc from $S$ to $V\setminus S$. 

A \emph{negative cost cycle} is an oriented cycle
$\gamma\in\{0,1,-1\}^E$ such that the cost $\l(\gamma) < 0$.  That a
flow $\f\in\Phi^+$ admits an augmenting cycle $\gamma$ means that the
flow $x+\delta\gamma$ is feasible, i.e.\ positive, for some scalar
$\delta>0$. It is well known (e.g.\ \cite{lawler76book}) that a
positive flow is of minimum cost if and only if it admits no
augmenting negative cost cycle.  Since the problem \eqref{e.primal}
does not have any upper capacities, a positive flow $\hat\f\in\Phi^+$
belongs to $\Hat\Phi$ if and only if for every negative cost cycle
$\gamma$ there is some arc $ij$ such that $\gamma_{ij}=-1$ and
$\hat\f\ij=0$.

From linear programming theory \cite{lawler76book}, we obtain that the
\emph{dual} problem corresponding to \eqref{e.primal} is
\begin{equation}\label{e.dual}
  \text{maximise $w=\p\T b$, where $\J\T\p\leq \l$}
\end{equation}
for the dual variables $\p=(\p_i)\in \RR^V$.  In the physarum solver,
the dual variables $\p_i$, $i\in V$, are interpreted as potential
values in the electric network.  It is the \emph{cocycle} of potential
differences, $\J\T\p = (\p_i-\p_j)_{ij\in E}$, that carry all
information and $\p\in\RR^V$ will only be determined up to a constant.
The cocycle space is the range of $\J\T: \RR^V\to\RR^E$ and
constitutes the orthogonal complement to the \emph{cycle space}
$\Phi(0,G)$ with respect to the standard inner product $(x,y)\to x\T
y$ on $\RR^E$.  Elements of the cycle space are also referred to as
\emph{circulations}.

Instead of working with cocycles, we mainly use the corresponding
\emph{fields} (``electrical field strengths''), or \emph{slopes}, by
which we mean vectors $\psi\in\RR^G$ such that $\l\psi$ is a cocycle
on $G$. In other words, fields are vectors of the form
$$ 
\psi = \Psi(\p) :=(\J\T\p)/\l 
    = \left(\frac{\p_i-\p_j}{\l\ij}\right)_{ij\in E}. 
$$
Notice that, the constraint in \qr{e.dual} can be written $\psi \leq
1$.  Fields and cocycles are for our purposes equivalent entities and
the potential can be recovered, as $\p_i = \p_k + \l(\psi\pi)$, where
$\pi$ is any oriented path from the (fixed) vertex $k$ to vertex $i$.

\subsection{The non-symmetric physarum solver}

Given a transshipment problem specified by $(G,b,\l)$ as above, an
electrical network is specified by the positive \emph{conductivity}
$\cy=(\cy\ij)\in\RR^E_+$.  The \emph{conductance} vector is then
$\cy/\l$ and the \emph{resistance} vector is given by $\l/\cy$.
Kirchhoff's equations can be stated using the weighted graph laplacian
$$ L(\cy) := \J \G \J\T,$$
where $\G =\diag(\sigma/\l)\in\RR^{E\times E}$ is the diagonal
matrix with the conductance vector $\cy/\l$ along the diagonal.  In
our setting, Kirchhoff's equations amounts to finding a solution
$\p\in\RR^V$ to the discrete Neumann problem
\begin{equation}
L(\cy)\p = b \label{e.kirchhoffs}
\end{equation}
A solution $\p$ gives the flow (the current) $\f$ via \emph{Ohm's law}
\begin{equation}\label{e.ohm}
  \f\ij =  \cy\ij\frac{\p_i-\p_j}{\l\ij} = \cy\ij \psi,
\end{equation}
where 
$$ \psi = \Psi(\p) = \J\T\p/\l. $$

In the Physarum Solver, we consider an electrical network which evolves
through ``time'' $t\in[0,\infty)$, where the state is specified by the
corresponding time-varying conductivity vector $\cy\t\in\RR_+^E$.  We
let $\f\t\in\RR^E$ and $\p\t\in\RR^V$ denote the corresponding current
and potential, which are derived from $G$, $b$, $\l$ and $\cy\t$ by
solving Kirchhoff's equations. The current $\f\t$ will at all times
constitute a flow in $\Phi(b,G)$, but not necessarily a positive flow,
nor will the potential vector $\p\t$ automatically be feasible for the
dual problem. Since $\supp\cy\t=G$ which is connected by assumption,
we can make the vector $\p\t$ unique by stipulating that its lowest
value is zero.

In this paper we use a version of the Physarum Solver,
where the conductivity vector $\cy\t = (\cy_{ij}\t)_{ij\in E}$ is
updated according to the non-linear equation
\begin{equation}\label{e.dicond2}
  \ddt \cy_{ij}\t + \cy_{ij}\t = \f_{ij}\t. 
\end{equation}
We can arbitrarily choose the initial condition as long as
$\cy_{ij}\t[0]>0$, for all $ij\in G$.  For fixed $\cy(0)>0$, $\l$ and
$b$, in this paper, we refer to the electrical network obtained by
letting $\cy\t$, evolve, for $t\geq0$, according to \eqref{e.dicond},
as the \emph{Physarum Solver}.
Ohm's law \eqref{e.ohm} immediately gives the following
alternative form of \eqref{e.dicond2}
\begin{equation}\label{e.dicond}
  \ddt \log\cy\t = \psi\t - 1,
\end{equation}
where $\psi\t := \Psi\left( \p\t \right)$.

The ``non-symmetric'' physarum solver defined by \eqref{e.dicond2}, or
equivalently by \qr{e.dicond}, is different from the
previous symmetric solver, presented in \cite{teroetalphysA} and
analysed in e.g.\ \cite{onishi}, where $|\psi\t| - 1$ and $|\f\ij\t|$
was used on the right hand sides of \eqref{e.dicond} and
\eqref{e.dicond2}, respectively.  The use of a non-symmetric
conductivity vector $\cy\ij\t$ is perhaps somewhat surprising; in an
electrical network, conductivity and conductance works symmetrically
at a fundamental level.  If we have a double (undirected) arc
$\{ij,ji\}$, the electrical network will have an effective
conductivity $\cy_{ij}+\cy_{ji}$ across the corresponding arc and the
flow across the two arcs will be proportional to these
conductances. The two terms will not evolve identically; if
the flow is from $i$ to $j$ consistently then the conductivity
$\cy_{ji}$ in the opposite direction will be of order
$\Ordo{e^{-t}}$ as $t\to\infty$.

\subsection{Duality analysis}

\begin{figure}[h]
  \centering
  \begin{tikzpicture}[scale=0.7,line cap=round,line
    join=round,>=triangle 45,x=1.0cm,y=1.0cm]
    \small \draw [color=gray!30,dash pattern=on 2pt off 2pt,
    xstep=1.0cm,ystep=1.0cm] (-1,-1) grid (10,6);
    \draw[->,color=black] (-1,0) -- (11,0); \foreach \f in
    {-1,1,2,3,4,5,6,7,8,9,10} \draw[shift={(\f,0)},color=black]
    (0pt,2pt) -- (0pt,-2pt) node[below] {};
    \draw[shift={(10,0)},color=black] (0pt,2pt) -- (0pt,-2pt)
    node[below] {$1$}; \draw[color=black] (-8pt,-2pt) node[below]
    {$0$}; 
    \draw[->,color=black] (0,-0.64) -- (0,6.05); \foreach \y in
    {-1,1,2,3,4,5,6} 
    \draw[shift={(0,\y)},color=black] (2pt,0pt) -- (-2pt,0pt) node[left] {}; 
    \draw[shift={(0,4)},color=black]
    (2pt,0pt) -- (-2pt,0pt) node[left] {$1$};
    \draw[color=black] (10.72,0.05) node [anchor=south west]
    {$\f\ij$}; \draw[color=black] (0.0,6) node [above]
    {$\psi\ij$};

    \coordinate (origo) at (0,0); \coordinate (A) at (0,4);
    \coordinate (B) at (10.5,4); \coordinate (C) at (0,-0.62);

    \draw [line width=2pt] (B)--(A)--(C);
    \draw [very thin, dotted] (-1,4)--(A);

    \draw (B) node[right] {``in-kilter line''};

    \foreach \nr/ \d / \r / \p in {1-/5/1.2/2} { 
      \draw (\d,0) node (D) {}; 
      \coordinate (yy) at (D |- A); 
      \coordinate (xx) at ($ (origo)!\r!(yy) $); 
      \coordinate (zz) at ($ (origo)!-.2!(yy) $); 
      \draw (xx) node[above] {$(\f\ij,\psi\ij)$};
      \draw[style=dotted] (xx.center) -- (yy.center); 
      \draw[thick] (zz.center) -- (xx.center); 
      \draw[style=dotted] (yy) -- (D);
      \fill[color=red] (xx) circle (4pt); 
      \fill[color=gray] (yy) circle (2pt) (D) circle (2pt); 
      \draw (D) node[below] (Z)
      {$\cy\ij$} (Z.center)+(2.5,0) node (ZP) {}; 
      \draw[->,thick] (Z) -- (ZP); 
    }

  \end{tikzpicture}

  \caption{The mechanism in the physarum solver as a primal-dual
    algorithm explained in a kilter diagram: If $\f\ij > \cy\ij$
    the conductivity $\cy\ij\t$ increases, from time $t$ to $t+dt$, by
    a factor $1+dt(\psi\ij-1)$.}
  \label{fig.kilter}
\end{figure}
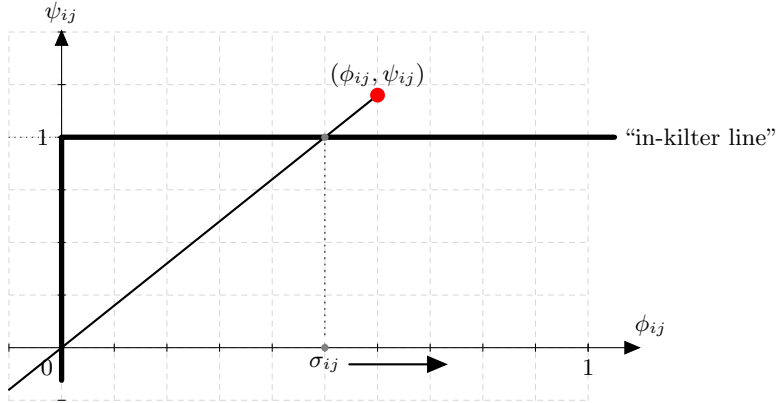

The convergence of the physarum solver to an optimal flow can be
discussed in the context of primal-dual methods (see
\cite{lawler76book}, \cite{MR1259413}) to solve mathematical
programming problems.  
By the Theorem of
Complementarity of Slackness (\cite[Theorem~13.4]{lawler76book}), a
primal-dual pair of feasible solutions $(\p,\f)$, corresponds to
optimal solutions of the dual and primal problems, respectively,
precisely when
\begin{equation}
  (\psi\ij - 1)\f\ij = 0,\qquad\forall\, ij\in E,\label{e.opt}
\end{equation}
which means that, for each arc $ij\in E$, the points $(\psi\ij,\f\ij)$
lie on the broken ``in-kilter line'' $\{0\}\times (-\infty,1] \cup
[0,\infty) \times \{1\}$ depicted in figure~\ref{fig.kilter}.  This
illustration of duality is used in e.g.\ \cite{dennis59}
\cite{minty60}, where it is showed that a non-linear resistor network
having the kilter-line as a ``characteristic curve'' for the resistor
along edge $ij$ will produce a current that solves the corresponding
minimum cost flow problem.

For the linear resistor networks used in the Physarum Solver, Ohm's
law, \eqref{e.ohm} ensures that the point $(\f\ij,\psi\ij)$ is on the
line going through the origin and the point $(\cy\ij,1)$.  The
optimality criterion \qr{e.opt} can therefore be reformulated as
\begin{equation}
  \text{$\psi\leq 1$ and $\f=\cy$.}\label{e.opt2}
\end{equation}
The dynamics of the physarum solver increase the logarithm of
conductivity if $\psi\ij>1$ and make it decrease if $\psi\ij<1$. From
this, we obtain the following observation.
\begin{prop}\label{p.fixobs}
  Any attractive fixed-point for the Physarum Solver must correspond
  to a primal-dual pair $(\f,\psi)$ where, for all $ij\in E$,
  $(\f\ij,\psi\ij)$ lie on the kilter line: $\psi\ij=1$, if
  $\f\ij>0$, and $\psi\ij\leq1$ if $\f\ij=0$. It will thus
  correspond to an optimal solution of the primal problem
  \eqref{e.primal} and the dual problem \eqref{e.dual}.
\end{prop}

\subsection{The primal convergence theorem}

However, it remains to show that $\cy\ij\t$ actually converges to some
attractive fixed point.  The optimal set is the subgraph $\HH$ of $G$
induced by arcs supporting some minimum cost flows, i.e.
$$\HH = G[\cup\{ \supp\ff: \ff\in \Hat\Phi(b,G)\} ]. $$ 
By the characterisation of optimal flows and the convexity of
$\Hat\Phi$, $\HH$ only support zero-cost cycles and it follows that a
positive flow is optimal precisely if it's support is contained in
$\HH$.
\begin{thm}\label{t.converge1}
  Assume that that the stated problem \eqref{e.primal} is
  feasible. Then the current $\f\t$ converges exponentially fast to
  some optimal flow $\hat\f\in\Hat\Phi$, i.e. $\f\t = \hat\f \pm
  e^{-\aOrdoOmega t}$.  The limit flow $\hat\f$ has moreover full
  support, i.e. $\supp\hat\f=\HH$.
\end{thm}

Note that, the fact that the Physarum Solver converges to a flow with
full support on the optimal set puts it on par with interior point
methods (see e.g. \cite{Roos2006}) in linear optimisation.  Although
there are some similarities, the Physarum Solver seems to be distinct
from interior point algorithms: For instance, neither the current
$\f\t$ nor the potential $\p\t$ are supposed to be feasible solutions
to \qr{e.primal} and \qr{e.dual} and are not interior points in this
sense.  In section \ref{s.discuss}, we briefly discuss possible
extensions of the Physarum Solver to more general linear and
non-linear optimisation problems.

\subsection{Infinity harmonic functions and the dual convergence
  theorem}

Given a directed graph $G=(V,E)$ and a subset $S\subset V$, we say
that a potential $\p\in\RR^V$ with corresponding field
$\psi=\Psi(\p)$, is a discrete \emph{$\infty$-harmonic function} on
$G\setminus S$, if for all $i\in V(G)\setminus S$, we
have
\begin{equation}
  \label{e.infharm}
  \max_{k: ki\in E} \psi_{ki} = \max_{j: ij\in E} \psi_{ij} \geq 0
\end{equation}
We can restate this condition as follows: Consider the \emph{level sets} 
$F_r:=\psi^{-1}(r)$, $r\geq 0$, of the field $\psi=\Psi(p)$ and define
$$ H_r:=\cup_{s>r} F_s = \psi^{-1}\left((r,\infty)\right). $$
Then $\p$ is $\infty$-harmonic on $F_0\cup H_0\setminus S$ if and only
if, for all $r$, $F_r$ is a union of directed paths with endpoints
contained in $S\cup V(H_r)$.

Discrete functions satisfying \eqref{e.infharm} have been studied
\cite{richmangames,auctionplay,tug-of-war09} as values for stochastic
two-person games which are symmetric in a certain sense.  The results
in \cite{tug-of-war09} concerns the more general concept of length
spaces. Our definition above is a little bit different, since we take
the maximum over in-going and outgoing arcs in \eqref{e.infharm},
separately. For the undirected problems, i.e.\ digraphs $G=(V,E)$
where $ji\in E$ whenever $ij\in E$, the definitions are equivalent.

Since all oriented cycles in $\HH$ have zero cost, it follows that
$\hat\psi=\diett_{\HH}$ is a field on $\HH$.  The following lemma is
an adaptation of Theorem~12 in \cite{auctionplay}.
For $A>0$, let 
\begin{lem}\label{l.infext}
  Assume that $\HH$ is \emph{connected} and that
  $\hat\p\in\RR^{V(\HH)}$ is the potential on $\HH$ given by
  $\Psi(\hat\p)=\diett_H$ and $\min\hat\p=0$.  
  Then there is a unique
  extension $\ip\in\RR^{V(\iH)}$ of $\hat\p$, where $\HH\subset \iH
  \subset G$, such that $\ip\rst{V(\iH)}$ is $\infty$-harmonic on
  $\iH\setminus\supp b$. It is the unique such extension that maximises  
  $$S(\ip,\iH,A) := \sum_{r\in\RR} |\Psi(p)^{-1}(r)| A^r. $$
  for all sufficiently large $A$.  Moreover, $\ip$ is dually feasible
  for \qr{e.dual}, i.e.\ $\Psi(\ip) \leq1$ for all arcs $ij\in G$.
\end{lem}
We defer the precise construction to
section~\ref{s.infharm}.  However, it should be noted that, if $G$ is
connected and comes from an undirected problem, i.e.\ if all arcs
comes in pairs of 2-cycles, then $\HH$ is connected and $V(\iH) =
V(G)$. Notice also that $\HH$ is connected if, for all $S\subset V$
such that the symmetric difference $S\bigtriangleup (\supp b) \not=\emptyset$,
we have $b(S)\not=0$; in particular,  this holds for a generic source vector $b$.

The following theorem states that $\p\t$ converges to a
$\infty$-harmonic function defined on $\iH$. 
\begin{thm}\label{t.infharm}
  Assume that $\HH$ is connected and \qr{e.primal} is feasible. Then
  $$
  \norm{\p\t\Rst{V(\iH)} - \ip\Rst{V(\iH)}}_\infty 
    = e^{-\aOrdoOmega t}. 
  $$
  where $\ip\rst{\iH}$ is the unique $\infty$-harmonic extension of
  $\hat\p$ obtained in \lemref{l.infext}.
\end{thm}

The connection between discrete $\infty$-harmonic functions and dual
solutions of the transshipment problem has, as far as we know, not
been noted elsewhere.  For the shortest path problem between a source
$s$ and a sink $t$, Dijkstra's algorithm construct a canonical
limiting dual solution given by the distance functions to $s$ (or
$t$).  The $\infty$-harmonic dual solution $\hat\p$ obtained is, in
contrast to these solutions, symmetric under the symmetry of
``time-reversal'', i.e.\ if we change the sign of $B$, $b$, we obtain
a new solution $-\hat\p + \operatorname{const}$.
In continuous theory, the relation between transport problems and the
Neumann problem for the $\infty$-laplacian has been noted e.g.\ in
\cite{neuinflap2}, but, as far as we understand, then only defined on
the transport set of optimal flows.  The physarum solver, if it can be
extended to the continuous setting, would perhaps be interesting for
such problems.

% One could ask what implications to modelling the finding of a discrete
% $\infty$-harmonic dual solution might be. The dual solution for a
% simple shortest path problem between a source and a sink has a meaning
% outside the limit set as follows: If an agent lands on a random node
% in $V(H^*)$ and then gets his/her mission, with equal probability, to
% transport itself either to the source or the sink and in doing so it
% should follow the gradient, either uphill or downhill, of some
% prescribed value function.  Then we conjecture that the best such
% function will be the dual $\p^*$ obtained above.

\subsection{Feasibility detection}

The physarum solver will detect an infeasible problem fairly quickly.
We state the following proposition without proof.
\begin{prop}\label{p.feas}
  If the problem \eqref{e.primal} is infeasible then $\|\p\t\|_\infty
  \to+\infty$ before time $t_0$ for some constant
  $t_0=t_0(b,\l,\cy(0))$.  Otherwise, if the problem \eqref{e.primal}
  is feasible then there is a constant $\p_{max}>0$ such that
  $\|\p\t\|_\infty < \p_{max}$.
\end{prop}

\section{Discussion}\label{s.discuss}

\subsection{General costs and general linear programs}

The duality analysis stated in Proposition~\ref{p.fixobs} remains
valid for the more general class of linear programs of the form
\qr{e.primal} given by more general coefficient matrices $\J=A$ and
linear cost $\l(\f)=\l\T\f$ where $\l>0$.  That is, provided the
laplacian matrix $L(\cy) = \J G\J\T$ is well behaved and provides a
solution, any stable fixed-point to \qr{e.dicond2} should correspond
to an optimal solution and vice versa. Although the proof in the next
section uses arguments specific to graphs, one can hope to extend the
applicability of the Physarum Solver to a larger class of linear
programming problems.

Proposition~\ref{p.fixobs}, can also be generalised to more general
convex costs, i.e., flows where the cost $C(\f)$ is a convex
increasing function of the flow $\f\geq0$. (See e.g. \cite{minty60} or
\cite{duffIIa}.)  In the argument one should then substitute $\l$ with
the gradient $\nabla_\f C(\f)>0$.  However, as is well known,
arguments based on duality breaks down for concave costs; these
problems are in general NP-hard \cite{pardalos90} and contains, e.g.\
the famous Steiner tree problem on graphs.  The actual models
\cite{TheoBiosci} of the physarum organism use non-linear functions
$|\f\ij|^{1+\gamma}$, $\gamma>0$, on the right hand side of
\eqref{e.dicond2}, to fit empirical data.  It may be taken as an
heuristic method for the minimisation with increasing concave
costs. The Physarum organism seems to do quite well solving simple
Steiner problem \cite{PASP}, which partly can be explained by such
concave cost minimising.

\subsection{Efficient implentation of the physarum model}

In this paper, we have not analysed the physarum solver as a computer
algorithm for transport problems. Thus we make no assertions about the
time or space complexity of an eventual computer
implementation. However, some observations regarding its eventual
place among existing algorithms can be made. It is clear that the
physarum solver does not rely on any centralised synchronous
computation and, assuming that the map that takes the conductivity
vector to the corresponding current vector and vector of potential
differences is provided by the environment, it is readily implemented
from local dynamics with simple rules.

The one important algorithmic complexity in the physarum solver as
defined here is the cost of solving Kirchhoff's equations for a given
conductivity vector. There are certainly efficient and localised
solvers for Kirchhoff's equations for general weighted laplacians (see
e.g. \cite{spielteng}). However, it may be more elegant and perhaps
ultimately more efficient to side-step this issue by using the
well-known (\cite{rwen} or \cite{kellybook79}) relations between
random walks on weighted graphs and electrical networks. As such, the
physarum solver could be implemented in an entirely decentralised
manner as a reinforced random walk.  There are algorithms for the
minimum cost flow problem, like the various auctions algorithm by
Bertsekas et.~al.\ (see \cite{bertstutor92}), that allows for being
implemented asynchronously and in parallel.

The physarum solver seems quite similar to the variant of Baum's
algorithm \cite{baumalg}, the ``knee-jerk algorithm'', used in
\cite{copperdoyle} to compute resistive inverses in electrical
networks --- if the obtained voltage over an edge is too large or
small, we increase or decrease the conductivity, quite oblivious of
global considerations. It would be nice to know if it is possible to
find a similar interpretation for the physarum solver. This could
possibly lead to faster implementations allowing for optimal steps in
the basic iterations.

Another motivation for studying a randomised physarum solver is to
model biological systems. A randomised physarum solver can take many
forms, for example as an ``ant algorithm'' with positive reinforcement
of walks. It should be noted that, unlike the prototypical
edge-reinforced ant algorithm \cite{ACO0}, the randomised distributed
physarum solver should update conductivities according to the net
transport rather than total transport across arcs.  In a forthcoming
paper, we plan to investigate the possibilities of different
distributed implementations of the physarum solver using the random
walk connection. Indeed, the physarum solver can be used as a unifying
model for several biological transport systems among them foraging
ants and, of course, the physarum organism.
%
%An intrinsic problem with the use of discrete graphs to model, say,
%ants is that the graph may not respect the underlying continuous
%geometry. For example, the rectangular lattice graph is not a good
%approximation for the shortest path problem in the plane, since we
%have shortest paths far away from those we would expect. We conjecture
%that one can remedy this by letting the conductivity vector slowly
%diffuse, according to the continuous geometry, on the line graph and
%thus keeping the conductivity tensor isotropic. It would be
%interesting to know if the physarum solver then would give a good and
%geometrically sound scaling limit.

%It should also be noted that the flow polyhedron $\Phi^+(b,G)$ has many
%interesting properties in relation to electrical networks. Every
%vertex is a flow supported on a tree in the (underlying undirected)
%graph $G$ and Kirchhoff's Matrix-Tree theorem give a direct
%correspondence between the electrical network, the random walk and the
%probability distribution on trees where each tree is weighted
%according to the product of the arc-conductances. Such random trees
%can be efficiently sampled with weights using Wilson's algorithm
%\cite{wilson96} and a variant of the randomised physarum solver
%outlined above could be to reinforce the conductances on the randomly
%sampled trees.

\section{The proofs}

In this section we prove \thmref{t.converge1} and \thmref{t.infharm}
and is divided in several subsections.  We start stating a couple of
lemmas, including \lemref{l.dcont} and \lemref{l.pnear}, concerning
the continuity of Kirchhoff's equations and the existence of solutions
having constant field-strength locally. The construction of the
infinity harmonic dual solution is made in
section~\ref{s.infharm}. Finally, we prove \thmref{t.converge1} and
later we prove \thmref{t.infharm} by induction, using the result from
\thmref{t.converge1} as the base case. 

\subsection{Some facts and lemmas}

Recall that, for $S\subset V$, $b(S)$ is the required net flow across
the cut $E(S,V\setminus S)$.  Let $\ka_{max} = \max_S |b(S)|$ and let
$\ka_{min} = \min\{ |b(S)| : b(S) \not= 0\}$.
The following elementary lemma states that if the total flow across a
set of arcs is small enough we can find a nearby
flow avoiding those arcs. 
\begin{lem}\label{l.removal}
  Let $\f\in\Phi$ be a given flow and let $E'\subset E$ be a set of
  arcs.  If 
  $$w:=\sum_{ij\in E'} |\f_{ij}| \leq \ka_{min}$$
  then there is a flow $\f'\in\Phi$ with support $\supp \f' \subset
  \supp\f\setminus E'$ and such that
  $$
  \|\f - \f'\|_\infty \leq w.
  $$
  Moreover, $\f\ij\f'\ij \geq 0$ for all $ij\in E$, so the flow $\f'$
  is in the same direction as that given by $\f$.
\end{lem}
\begin{proof}[Proof of \lemref{l.removal}]
  It is enough to show this for the case $|E'|=|\{ab\}|=1$, the
  general case follows easily by induction. We can assume that the
  flow $\f$ is positive, since we can reorient the graph $G$ without
  changing the statement of the lemma.

  Let therefore $\ff = \f - \f_{ab}\cdot \ett_{ab}$ where
  $\f_{ab}=w>0$. We intend to repair $\ff$ by constructing a $ab$-flow
  $\delta$ of value $w$ such that $\ff=\ff+\delta$ does not change the
  orientation of any arc.  By the Max-Flow-Min-Cut Theorem and
  Theorem~4.3 in \cite{lawler76book}, this flow problem can be solved
  with $\|\delta\|_\infty \leq w$ unless there is an cut
  $(S,V\setminus S)$ in $G$ separating $a\in S$ and $b\in V\setminus
  S$ with the following properties: Firstly, no arc in $G$, except
  $ab$ goes from $S$ to $V\setminus S$. Secondly, the capacity of the
  cut, which is given by $c:=\sum_{i\in V\setminus S,j\in S} \phi_{ij}$, satisfies
  $c < w$. 

  Hence, for the original flow $\f$, the flow goes from $V\setminus S$
  to $S$ for all arcs in the cut except for the arc $ab$. Since the
  net flow of $\f$ across the cut equals $c-w < \ka_{min}$, we can
  conclude that $b(S) = 0$.  But then $c=w$ contradicting the
  existence of the cut.
\end{proof}

The set $\BFS$ of \emph{basic feasible solutions} (\tbfs) are positive
flows with support on directed graphs without oriented cycles, i.e.\
the underlying undirected graphs are forests.  For a $\bfs\in\BFS$,
removing an edge $ij$ from the support of $\bfs$ results in a graph
with two new components and $\bfs\ij = b(S)>0$, where $S$ is the
component containing $i$. Hence, we have
\begin{equation}
  \label{e.bfsmin}
  \ka_{min} \leq \bfs\ij \leq \ka_{max}, \quad \forall\, ij\in \supp\bfs.
\end{equation}

A positive flow $\f\in\Phi^+$ is \emph{circulation free} if it is
supported on a subgraph without a directed cycle.  A circulation free
positive flow $\f$ with minimal edge support is a \tbfs; if
$\g\in\{-1,+1,0\} $ is an
oriented cycle with $\supp\g\subset\supp\f$ then $\f'=\f-\alpha \g$ is
a positive circulation free flow with $\supp\f'\subsetneq\supp\f$ for
$\alpha = \min\{\f\ij/\gm\ij: \gm\ij>0\}$.  It follows that, for any
circulation free positive flow $\f$, we can find a $\bfs\in\BFS$ and a
$c>0$ such that $c\bfs\ij\leq\f\ij$, where equality holds for at least
one edge. It is then easy to deduce that $\f$
can be written as a convex combination $\f = \sum_\bfs c_\bfs \bfs$,
where at most $|\supp\f|$ coefficients $c_\bfs$ are non-zero.  We also
have the following lemma.
\begin{lem}\label{l.corrij}
  Let $\f$ be a circulation free positive flow and $ab$ an edge in the
  support of $\f$. There is a $\bfs\in\BFS$ with the property that
  $\bfs\ij[ab]>0$ and such that
  $$ \min \f(\supp\bfs) \geq
  \frac{\ka_{min}}{\ka_{max}} \cdot \frac{\f\ij[ab]}{|\supp\f|}. $$
\end{lem}
\begin{proof}
  Write $\f$ as a convex combination $\f = \sum_\bfs c_\bfs \bfs$ of
  \tbfs{}s where at most $|\supp\f|$ of the $c_\bfs$s are
  non-zero. Hence, 
  $$ \max\{c_\bfs\bfs\ij[ab]: \bfs\ij[ab] > 0 \}\geq
  \frac{\f\ij[ab]}{|\supp\f|} \implies c_\bfs \geq \frac 1{\ka_{max}}
  \cdot \frac{\f\ij[ab]}{|\supp\f|}. $$ The stated inequality then
  follows, since positiveness implies that $\f\geq c_\bfs\bfs$, for all
  $\bfs$, and thus
  $$ \min \f(\supp\bfs) \geq c_\bfs \cdot \min\bfs(\supp\bfs)\geq
  c_\bfs\ka_{min} $$
  by \qr{e.bfsmin}.
\end{proof}

\subsection{Kirchhoff's equations}

Given a weighted graph $G=(V,E,\l)$ and a conductivity vector
$\cy=(\cy\ij)>0$, we will consider Kirchhoff's equation with flow
requirements (a Neumann problem)
\begin{equation}
  \label{e.kneu}
  L(\cy) \p = b, \quad b\in\RR^V
\end{equation}
and also with voltage prescriptions (a Dirichlet problem)
\begin{equation}
  \left(L(\cy) \p\right)\big\vert_{V\setminus S} = 0,
  \quad  \p\vert_{S}= q\vert_{S}. 
  \label{e.kdir}
\end{equation}
The orientation of the edges of $G$ are inessential.  The equation
\eqref{e.kdir} has a unique solution as long as $S$ intersect each
component of $G$.  The equation \eqref{e.kneu} has a solution as long
as $b(W)=0$ for every component vertex-set $W$ of $G$ and is unique up
to linear combinations of the indicators of component vertex-sets of
$G$. In particular, the corresponding cocycles and fields are unique.
We say that $\p\in \RR^V$ is \emph{harmonic} outside $S$ if $L(\cy)=b$
and $\supp b \subset S$.

The set of solutions to \eqref{e.kdir} are denoted $\mc D( \cy\vert_G,
q\vert_S)$. The solution space to problem \eqref{e.kneu} are similarly
written $\mc N( \cy\vert_G, b\vert_S)$, where $\supp b \subset S$. For
definiteness, we assume the disambiguation rule that the lowest value
for $\p$ in each component of $G$ is zero and we write $\p=\mc
D(\cy\vert_G,q\vert_S)$ and $\p=\mc N(\cy\vert_G,b\vert_S)$ to stress
that the solutions are assumed to exist uniquely. The vector $\cy$
should have domain containing $G$ and $q$ should be defined on a
subset of $V(G)$. For the problem \qr{e.kneu} we implicitly extend $b$
to $V(G)$ by setting $b$ to zero outside the singular set $S$ if
needed.

A conductivity vector which at the same time is a minimum cost flow
allow solutions to \qr{e.kdir} and \qr{e.kneu} where the corresponding
field is constant.
\begin{lem}\label{l.pnear}
  Assume $\cy\in\RR_+^{H}$ is a conductivity vector supported on a
  graph $H$ where all oriented cycles satisfies $\l(\g)=0$. Assume
  further that $S$ is a subset $S\subset V(H)$ and that $b \in \RR^V$
  is an admissible source vector such that $\supp b \subset S$.  If
  $\cy$ is a positive flow with $\J\cy = b$ then
  \begin{enumerate}[(a)]
  \item\label{item.a} the solution $\p=\mc N(\cy\vert_{H},b\vert_S)$
    is given by $\Psi(\p)\equiv1$.
  \item\label{item.b} if $\Psi(\p) \equiv r$, $r>0$, then $\p=\mc
    D(\cy\vert_{H},\p\vert_S)$.
  \end{enumerate}
\end{lem}
\begin{proof}[Proof of \lemref{l.pnear}]
  By assumption, $\l$ is orthogonal to all
  oriented cycles and is thus a cocycle. Hence, there is a vector
  $\p\in\RR^{V(H)}$ such that $\Psi(\p)\equiv1$.  The first statement in
  the lemma then follows since
  $$ L(\cy)\p = \J \cdot \diag(\cy) \cdot \Psi(\p) = \J\cy=b. $$
  The potential $\p$ is essentially unique since $\supp b$ must
  intersect all components of $H=\supp\cy$. 

  The second statement follows by the same reasoning since $\supp b
  \subset S$ and the obtained $\p$ is thus harmonic outside $S$. 
\end{proof}

The following elementary, but crucial, lemma expresses the continuity
of the solutions in Kirchhoff's equations.  Assume $\cy$ and $\Hat\cy$
are two conductivity vectors on the connected digraph $G=(V,E)$ where
we assume that $\supp\cy = E$ and where $H\subset G$ is the graph
induced by $\supp\hat\cy$.  Consider two cases
\begin{description}
\item[(D)]\label{caseD} $\p=\mc D(\cy\vert_G,q\vert_S)$ and
  $\hat\p=\mc D(\hat\cy\vert_H, q\vert_{S})$, $S\subset V(H)$. 
\item[(N)]\label{caseN} $\p=\mc N(\cy\vert_G,b)$ and $\hat\p\in \mc
  N(\hat\cy\vert_H, b\vert_{V(H)})$, where $S:=\supp b \subset V(H)$.
\end{description}
\begin{lem}\label{l.dcont}
For the solutions $\p$ and $\hat\p$ in both cases above, define 
$$\tl\p:=\mc D(\cy\vert_G,\hat\p\vert_{V(H)}). $$
Let
\begin{equation*}
  \delta(\cy,\Hat\cy) :=
  \frac{\|\cy-\Hat\cy\|_\infty}{\minp\hat\cy}.
\end{equation*}
where $\minp (x_\alpha):=\min\{ x_\alpha: x_\alpha>0\}$.
Then
\begin{equation*}
  \|\p-\tl\p\|_\infty =
  \Ordo{\delta\,\|\hat\p\|_\infty}
\end{equation*}
as $\delta\to0$.
\end{lem}

\subsubsection{Proof of \lemref{l.dcont}}

Let $\psi=\Psi(\p)$, $\tl\psi=\Psi(\tl\p)$ and $\hat\psi=\Psi(\hat
p)$. Notice that $\tl\psi\ij = \hat\psi\ij$, for $ij\in H$.  Let
$r:=\p-\tl\p$, and $\rho:=\Psi(r) = \psi-\tl\psi$.

We expand $\cy\tl\psi^2 = \cy(\psi-\rho)^2$ and obtain
\begin{equation}
  \label{e.xdir1}
      \l(\cy\tl\psi^2) = \l(\cy\psi^2)+\l(\cy\rho^2)-2\l(\rho\cy\psi).
\end{equation}
Similarly, since $\psi=\tl\psi+\rho$, we obtain
\begin{equation}
  \label{e.xdir2}
      \l(\hat\cy\psi^2) =
      \l(\hat\cy\tl\psi^2)+\l(\hat\cy\rho^2)+2\l(\rho\hat\cy\tl\psi), 
\end{equation}
Let $\e=\hat\cy-\cy$.  Adding \qr{e.xdir1} to \qr{e.xdir2}, and
rearranging the terms gives
\begin{equation}\label{e.xdir3}
\l\left(\e(\psi^2-\tl\psi^2)\right) =
  \l\left((2\hat\cy-\e)\rho^2\right) +
2 \l\left(\rho(\f-\hat\f)\right)
\end{equation}
where $\f= \cy\psi$ and $\hat\f = \hat\cy\tl\psi$. Both $\f$ and
$\hat\f$ are currents with sources contained in $S$ and the
last term in \qr{e.xdir3} 
$$ \l\left(\rho(\f-\hat \f)\right) = r\T \, \J\,(\f-\hat\f) = 0 $$
vanishes. To see this, note that the Dirichlet case (D) implies that
$r=\p-\tl\p$ is zero on $\supp \J(\f-\hat\f) \subset S$.
In the Neumann problem (N), $\J(\f - \hat\f)= b-b=0$, since $\f$ and
$\hat\f$ are flows with source vector $b$, by assumption.

Hence, after some more rearranging and using that $\psi^2-\tl\psi^2 =
\rho(\tl\psi+\psi)$ we obtain from \qr{e.xdir3} the equality
\begin{equation}\label{e.xdir4}
\l(\hat\cy\rho^2) = 
\frac12 \l\left(\rho\e(\tl\psi+\psi + \rho)\right) =
\l(\rho\e\psi)
\end{equation}

Since $r$ equals $\p-\hat\p$ on $V(H)$ and since $r$ is harmonic
outside $V(H)$, we have $\norm{r}_\infty = \norm{r\Rst{V(H)}}_\infty$
by the maximum principle (e.g.\ \cite{rwen}).  Moreover, since
$\rho=\Psi(r)$ and since, by assumption, $r_k=0$ for some vertex $k$
in each component of $H$, we have, for any $i\in V(H)$ that
$$  r_i= r_i-r_k = \l(\pi_i\rho), $$
where $\pi_i$ is a vector representing some oriented path connecting
$i$ with the grounded vertex $k$.  The Cauchy-Schwarz inequality then
implies that
\begin{equation}
  \norm{r}_\infty^2 =
  \max_{i} \left(\l(\rho\pi_i)\right)^2 \leq 
  \diam H \cdot \l\left(\rho^2\rst{H}\right),
  \label{e.cspi}
\end{equation}
where $\diam H$ denotes the length of a longest shortest
path in $H$.

It follows from \qr{e.cspi} that the left hand side in \qr{e.xdir4}
$$
\l(\hat\cy\rho^2) \geq (\minp\hat\cy) \cdot \l\left(\rho^2\rst H\right) 
\geq \frac{\minp\hat\cy}{\diam H} \cdot \|r\|^2_\infty. 
$$
By Cauchy-Schwarz, the right hand side of \qr{e.xdir4}
$$
\l(\rho\e\psi) 
\leq \frac{\normi\e}{\min\l} \cdot |G| \cdot 
\normi{\l\rho}\cdot\normi{\l\psi},
$$
which gives 
$$
\l(\rho\e\psi) 
\leq \frac{4 |G|}{\min\l} \cdot  \norm{r}_\infty\cdot \norm{\p}_\infty
\cdot 
\normi{\e},
$$
since $\normi{\l\Psi(u)} \leq 2\normi{u}$.  

Hence, by comparing sides in \qr{e.xdir4}, we obtain, after
dividing both sides by $(\diam H)/(\norm{r}_\infty \minp\hat\cy)$,
that
\begin{equation}
  \label{e.final}
  \normi{r} \leq
  \left(4|G|\cdot\frac{\diam H}{\min\l} \right)\cdot\delta\cdot \|\p\|_\infty. 
\end{equation}
where $\delta = \|\e\|_\infty/\minp\hat\cy$. 

It just remains to verify that $\|\p\|_\infty =
\Ordo{\|\hat\p\|_\infty}$:  In case (D) it follows immediately that
$\norm{\p}_\infty=\norm{\hat\p}_\infty$ by the maximum principle. In
case (N), we have $\cy \geq (1-\delta)\hat\cy$. The solution to the Neumann
problem with conductivity $(1-\delta)\hat\cy$ is given by
$\hat\p/(1-\delta)$ and we deduce from Rayleigh's monotonicity
principle (\cite{rwen}) that 
$\norm{\p}_\infty \leq \norm{\hat\p}_\infty/(1-\delta)$.
\qed

\subsection{The construction of $(\iH,\ip)$ in \lemref{l.infext}}
\label{s.infharm}

Given a pair $(H,\p)$, where $H$ is a subgraph of the weighted digraph
$G=(V,E,\l)$ and $\p\in\RR^{V(H)}$ is a potential on $H$, we say that
a \emph{directed path} $\pi\subset G$ with endpoints
$\{a,b\}\subset V(H)$ has \emph{($\p$-) slope}
$r(\p;\pi)=({\p_a-\p_b})/{\l(\pi)}$.  Note that, the slope along a
directed path in $H$ is a \emph{weighted average} of the slopes of
arcs along the path, i.e.\
\begin{equation}
  \label{e.meanslope}
  r(\p;\pi) = M_\pi(\psi) := \frac1{\l(\pi)}\, \l(\pi\psi)
\end{equation}
where $\psi = \Psi(\p)$ is the field on $H$ corresponding to $\p$. We
use $\pi$ to both denote the directed path and the corresponding
vector $\diett_\pi\geq 0$ in $\RR^E$. 

The $\p$-slope is \emph{constant} on the path $\pi\subset H$ if $\psi\ij =
r(\p;\pi)$ for all arcs $ij\in \pi$.  A \emph{directed} $kl$-path
$\pi\subset G$ is a directed \emph{trajectory} to a connected subgraph
$H\subset G$ if both endpoints, $k$ and $l$, of $\pi$ belongs to
$V(H)$ and all internal nodes are in $V(G)\setminus V(H)$.

To construct the graph $\iH$ and $\ip$, we consider a given pair
$(H,\p)$ where $H\subset G$ is a connected subgraph and
$\p\in\RR^{V(H)}$ is a potential on $H$. Initially, we have
$(H,\p)=(\HH,\hat\p)$ and we iterate the elementary extension
$(H',\p')$ of $(H,\p)$ defined as follows: Let $\pi$ denote a directed
trajectory to $H$ having \emph{maximum} positive slope
$r=r(\p,\pi)>0$. Let $H'=H\cup\pi$ and extend $\p$ to $\p'$ by
requiring that $\p'$ has constant slope on $\pi$. In the case we have
many paths of the same maximum slope, we pick the first path in some
lexicographic order.  If no trajectory of positive slope exists then
we stop and set $\iH=H$ and $\ip\vert_{V(\iH)}=\p$.

The construction ensures that the slope of extension paths will
decrease in value: If $(H',\p')$ has a trajectory $\pi'$ such that the
$\p'$-slope exceeds $r$ then $\pi$ can not be a trajectory of maximum
slope to $(H,p)$. To see this, let $\pi''$ be the unique trajectory to
$H$ that contains $\pi'$ and note that the $\p$-slope of $\p''$ equals
it $\p'$-slope; which is obtained as an average of the $\p'$-slope
along sub-paths of $\pi$ and $\pi'$ and would thus exceed the
$\p$-slope of $\pi$, contradicting the choice of $\pi$ as a trajectory
of maximum slope.

Furthermore, the $\hat\p$-slope along any $\HH$-trajectory, $\pi$,
with endpoints $a,b\in V(\HH)$, must be strictly less than
$1$. Otherwise $\pi$ can be extended to an oriented cycle $\g$, with
$\pi$ as a positively oriented subsegment and $\g\setminus\pi$ an
oriented $ba$-path in $\HH$. Thus,
$$ \l(\g) = \l(\pi)+\l(\g\setminus\pi) = \l(\pi) - (\hat\p_a - \hat\p_b). $$ 
If $r(\hat\p;\pi)\geq 1$ then $\l(\g)\leq 0$ and $\g$ is an augmenting
negative- or zero-cost cycle contradicting the choice of $\HH$.

The constructed extension clearly maximises $S(\ip,\iH,A)$ for large
$A$, since it is easy to see that $F_r = (\ip)^{-1}(r)$ contains
\emph{all} edges contained in some
$H_r=(\ip)^{-1}(r,\infty)$-trajectory of slope $r$: If we inductively
assume that the extension to $H_r$ is maximising --- which is trivial
when $H_r=\HH$ --- then, so, is the extension to $H_r\cup F_r$, since
$r$ and then $|F_r|$ are maximal.

\subsection{The proof of \thmref{t.converge1}}

\subsubsection{Some relations and identities}
We first state some more equivalent formulations of the adaptation
rule given by \qr{e.dicond2} and \qr{e.dicond}: By multiplying
\eqref{e.dicond2} with the integrating factor $e^t$ and then
integrate, we obtain
\begin{equation}
  \cy\t = (1-e^{-t})\tl\f\t + e^{-t} \cy\t[0],\label{e.dicond3}
\end{equation}
where $\tl\f\t\in\Phi$ is the discounted time-averaged flow
\begin{equation}
  \tl\f\t := 
  \frac1{1-e^{-t}} \cdot \int_{0}^t \f(s) \, e^{-(t-s)}\, \d s.
  \label{e.tldef}
\end{equation}
That $\tl\f\t\in\RR^E$ is a flow follows from the convexity of
$\Phi(b,G)$.

If we integrate \qr{e.dicond}, we derive
\begin{equation}
  \label{e.dicond5}
  \cy\t = \cy\t[0]\, e^{-(\bar\psi -1)t}.
\end{equation}
where
\begin{equation}\label{e.mdvdef}
    \bar\p\t := \frac1t \int_0^t \p\t[s]\d s,
\end{equation}
and $\bar\psi\t=\Psi(\bar\p\t)$ is the time average of $\psi\t$.

For a flow obtained as a current in an electrical network, the flow is
always from higher potential to lower potential. Thus, if $\f$ is a
current then each arc $ij$ is contained in the cut given by $S=\{k:
\p_k \geq \max\{\p_i,\p_j\} \}$ and the flow across this cut is from
$S$ to $V\setminus S$ only. Hence,
\begin{equation}
  |\f\ij| \leq \ka_{max} 
  \label{e.ibounded}
\end{equation}
Since $\tl\f$ in \qr{e.dicond3} is an average of the currents $\phi(s)$,
$s\leq t$, we deduce that conductivities stay bounded
\begin{equation}
  \cy\t = (1-e^{-t})\tl\f\t + e^{-t}\cy\t[0] 
 \leq K_0 := \ka_{max} + \max\cy\t[0].
\label{e.gbounded}
\end{equation}
(We will label constants $K_1,K_2,\dots$ and $c_1,c_2,\dots$.)

Since $\cy\ij\t> 0$, we also obtain that
\begin{equation}
  \label{e.almostpos}
  \min_{ij} \tl\f_{ij}\t  \geq -\sigma(0)\cdot \frac{e^{-t}}{1-e^{-t}},
\end{equation}
which means that the flow $\tl\f\ij\t$ is ``almost positive'' for
large $t$.  Let therefore $\ff\t$ be the non-circulatory positive flow
$\ff\t\in\Phi^+$ obtained from $\tl\f\t$ by the application of
\lemref{l.removal} with respect to arc-set
$$ E:=\{ ij: \tl\f_{ij}\t \leq 0 \} \cup \{ij: \bar\psi\ij\t \leq
0\}. $$ The flow $\ff\t$ is non-circulatory, since $\bar\psi\ij\t<0$
for at least one $ij$ along any directed cycle.  \lemref{l.removal} in
conjunction with \eqref{e.almostpos} and \qr{e.dicond5} gives that
$\ff\t = \tl\f\t + \Ordo{e^{-t}}$ and we deduce from \qr{e.gbounded}
that
\begin{equation}
    \cy\t = \ff\t + \Ordo{e^{-t}}. \label{e.xphi}
\end{equation}
Hence, there is a constant $K_1$, such that 
\begin{equation}
  \label{e.kphip}
  \frac12 \ff\ij\t \leq \cy\ij\t \leq 2 \ff\ij\t 
  \quad
  \text{whenever $\cy\ij\t \geq K_1 e^{-t}$.}
\end{equation}

The observation \qr{e.gmean} below is in some sense the
key to the proof of \thmref{t.converge1} and the statement is
essentially the same as that given in by Onishi et al.\ in
\cite{onishi}. Let
\begin{equation}
\label{e.lcydef}
\lcy\t := \log\left(\cy\t/\cy\t[0]\right).
\end{equation}
Let $\g=\gp-\gm \in\RR^E$ be fixed.
Applying
$\l(\cdot \g)$ on both sides of \qr{e.dicond5} gives
\begin{equation}
  \ls{\lcy\t}\g = \left(\ls{\bar\psi\t}\g - \l(\g)\right)\, t
  -\l(\g)\, t. 
  \label{e.fintegr}
\end{equation}
After dividing by $-\frac1{\l(\gm)}$ and rearranging, we obtain
\begin{equation}
  \label{e.gmean}
    M_{\gm}(\lcy\t) = r \cdot M_{\gp}(\lcy\t) + \left(r-1\right)t 
    - t \frac{\ls{\bar\psi\t}\g}{\l(\gm)},
\end{equation}
where $r(\g):={\l(\gp)}/{\l(\gm)}$ and $M_\g(x)$, $\g\geq 0$, is the
weighted mean
$$ M_\g(x) := \frac1{\l(\g)}\, \l(\g x). $$

If $\g\in\Phi(0,G)$ is a circulation, then $\ls{\bar\psi\t}\g = 0$ and we
obtain
\begin{equation}
  \label{e.gmean0}
    M_{\gm}(\lcy\t) = r \cdot M_{\gp}(\lcy\t) + \left(r-1\right)t 
\end{equation}
and if $\g=-\pi$, where $\pi$ is a directed $ab$-path, we obtain
\begin{equation}
  \label{e.gmean1}
    M_{\pi}(\lcy\t) =  t (M_\pi(\bar\psi\t)-1) = 
    t\left(\frac{\bar\p_a-\bar\p_b}{\l(\pi)} - 1\right)
\end{equation}

\subsubsection{The proof of \thmref{t.converge1}}
Let $kl\in G\setminus \HH$. Using \qr{e.gmean0}, we show below, for
some $a_1>0$, that 
\begin{equation}\label{e.wmax}
\cy\kl\t \leq \Ordo{e^{-a_1t}}, \quad\forall\, kl\in G\setminus\HH
\end{equation}
It follows from \qr{e.wmax} and \lemref{l.removal} that we can write
the circulation free positive flow $\ff\t$ as
$\ff\t=\hat\ff\t+\Ordo{e^{-a_1t}}$, where $\hat\ff\t$ is an optimal
flow supported on $\HH$. (Recall that any such flow is optimal.) 

Hence we have from \qr{e.xphi} the decomposition
$$ \cy\t = \hat\ff\t + \Ordo{e^{-a_1t}}. $$
We also
show below that
\begin{equation}
  \label{e.hathbelow}
  \hat\ff\ij\t = \OrdoOmega1,\quad \forall\, ij\in\HH. 
\end{equation}

The conditions of \lemref{l.dcont} are now fulfilled for $\cy=\cy\t$
and $\hat\cy = \hat\ff\t$. Moreover, \lemref{l.pnear} implies that the
solution $\hat\p=\mc{N}(\hat\ff\t\vert_{\HH},b\vert_{\HH})$ is
constant and given by $\Psi_{\HH}(\hat\p)\ij = \ipsi\ij = 1$,
$ij\in\HH$.  The estimates \qr{e.wmax} and \qr{e.hathbelow} shows that
we may take $\delta(\hat\ff\t,\cy\t)=\Ordo{e^{-a_1t}}$ in
\lemref{l.dcont}. Thus, we obtain from \lemref{l.dcont}, that $\forall\,
i\in V(\HH)$, $\p_i\t = \hat\p_i + \Ordo{e^{-a_1t}}$ or,
equivalently that 
\begin{equation}
  \psi\ij\t = \ipsi\ij + \Ordo{e^{-a_1t}}, 
\quad\forall\, ij\in\HH,
\label{e.pconvH}
\end{equation}
where $\ipsi\ij = 1$ for $ij\in\HH$. 

Write 
$$\psi\ij\t = \ipsi\ij + \e\ij\t.$$ 
As a direct consequence of \qr{e.dicond5}, we have
\begin{equation}
  \cy\ij\t = (1-\tl\e\ij\t)\cdot\cy\ij^* \cdot e^{(\ipsi\ij-1)t}
  \label{e.cyprec}
\end{equation}
where $\cy^*>0$ is the constant vector
$$ \cy^* := \cy\t[0] \exp\left(\int_0^\infty \e\t[s]\, ds\right). $$
and 
$$ \tl\e\t = \exp\left\{ \int_t^\infty \e\t[s]\d s \right\} - 1, $$
since the integrals are convergent by \qr{e.pconvH}. We obtain from
\qr{e.cyprec}, for $ij\in\HH$, that
$$ 
\lim_{t\to\infty} \cy\t =
\lim_{t\to\infty} \hat\ff\t = \hat\f \in \Hat\Phi. 
$$
and that $\lvert\cy\t - \hat\f\rvert = \Ordo{e^{-a_1t}}$.

For $ij\in\HH$, the exponential convergence of 
$$\lim\f\ij\t = \lim(\cy\ij\t\psi\ij\t)=(\lim\cy\ij\t)\cdot 1 =
\hat\f\ij, $$ 
follows. For $ij\not\in\HH$, we have on account of Ohm's law and 
\eqref{e.wmax} that
$$ |\f\ij\t| \leq K_3\t\cdot\cy\ij\t  = \Ordo{e^{-a_1t}}, $$
where, due to the maximum principle, we can take 
$$ K_3\t := \frac{\max_{i,j\in\supp b} |\p_i\t - \p_j\t|}{\min \l}. $$ 
The numerator is stays bounded and is for large $t$ at most $\diam\HH$
since $\psi\t$ converge to $1$ on $\HH$.   
\qed

\begin{proof}[Proof of \qr{e.wmax}]
  Using \lemref{l.corrij} on the circulation free flow $\ff\t$ and
  taking the \qr{e.xphi} into account, we deduce the following: There
  is a constant $c_3$ such that, for any given edge $ij$, either
  $\cy\ij < K_1 e^{-t}$ or else there is a $\bfs=\bfs\t\in\BFS$ such
  that $\bfs_{ij} > 0$
  \begin{equation}
    \min\cy(\supp\bfs) \geq c_3 \,\cy\ij\label{e.crj}.
  \end{equation}
  (We suppress the time variable for readability.)

  As before, let $\lcy\t=\log\left(\cy\t/\cy\t[0]\right)$.  With
  $kl\in G\setminus\HH$ as in \qr{e.wmax}, either $\cy\kl\t\leq
  K_1e^{-t}$ in which case we are through. Otherwise, there is a \tbfs
  $\bfs_1$ such that
\begin{equation}
 M_{\bfs_1}(\lcy\t) \geq \min \lcy(\supp \bfs_1) \geq \log \cy\kl\t 
 + \log K_4,\label{e.crj2}
\end{equation}
where $K_4 = c_3/\max\cy\t[0]$, with $c_3$ as in \qr{e.crj}.   

Let $\BOS$ denote the set of basic optimal solutions, i.e.\ feasible
solutions supported on forests contained in $\HH$.  If $\bfs_2\in\BOS$
then we can apply \qr{e.gmean0} on the circulation $\g=\bfs_2-\bfs_1$
and deduce from \qr{e.crj2} that
\begin{equation}
  \label{e.bfscmp}
  \log \cy\kl\t \leq r M_{\bfs_2}(\lcy\t) - (1-r)t - \log K_4,
\end{equation}
where
$$ r = \l(\bfs_2)/ \l(\bfs_1) \leq 
\max_{\bfs\not\in\BOS} \frac{\l(\bfs_2)}{\l(\bfs)} =: 1-a_1 < 1.$$
Moreover, from \qr{e.gbounded}, we know
that $M_{\bfs_2}(\lcy\t) \leq \log K_0$, whence
$$ \log \cy\kl\t  \leq (1-a_1)\log K_0 - a_1t - \log K_4, $$
which proves \qr{e.wmax}.
\end{proof}

\begin{proof}[Proof of \qr{e.hathbelow}]
  Clearly, we have for all $t>0$ that $\max_{ij} \ff\ij\t \geq
  \ka_{max}/|E|$, and hence, by \qr{e.crj}, there is, for all $t$,
  some $\bfs_0=\bfs_0\t\in\BFS$ such that
$$ M_{\bfs_0\t}(\lcy\t) \geq \log\left(\ka_{max}/|E|\right)  + \log c_3 -
\max\log\cy(0) =: \log c_6. $$ 
Choose $\bfs_1\in\BOS$ arbitrary. If we apply \qr{e.gmean0} on the
circulation
$\g=\g\t=\bfs_0\t-\bfs_1$ we get
$$ M_{\bfs_1}(\lcy\t) \geq r \cdot \log c_6 + (r-1)t \geq \log c_6.  $$ 
since, in this case $r\geq 1$ on account of $\bfs_1$ being optimal. Since
$\lcy\ij\t\leq \log K_0$, for all $ij$, by \qr{e.gbounded}, it follows
that
$$ 
\min \lcy(\supp\bfs_1) \geq \frac{\l(\bfs_1)\cdot\left(\log c_6-\log
    K_0\right)}{\min\l}.
$$
This proves \qr{e.hathbelow}, since $\bfs_1$ was chosen arbitrary
among basic solutions with support in $\HH$.
\end{proof}

\subsection{The proof of \thmref{t.infharm}}

Since $\HH$ is connected, the construction of $\ip$ as in section
\ref{s.infharm} goes through. 
Let
\begin{equation}
\ipsi(\iH) = \{r_0,r_1,\dots,r_M \} \subset (0,1],\label{e.rkdef}
\end{equation}
denote the slopes obtained by $\ip$ on $\iH$, where $1=r_0>r_1>\cdots
>r_M >0$. Let further $r_{M+1}=0$.  For $r=r_k>0$, let
$F_{r}=(\ipsi)^{-1}(r)$, $H_r:=\cup_{s>r} F_s$ and $F_0:=G\setminus
\iH$ and $H_1$ the empty graph vertex-set $\supp b$.  Note that
$F_1=\HH$ and that $H_r$ is connected for $r<1$ and that $H_0=\iH$.

Assume that $r=r_k\in\ipsi(\iH)$ is fixed. We simplify the notation by
putting $H=H_r$, $F=F_r$. We let $r'=r_{k+1} < r$ and put
$H'=H_{r'}=H\cup F$.  We prove \thmref{t.infharm} with an induction
argument over $k=1,2,\dots$ using the induction hypothesis that, for
some constants $a_k>0$ and $C_k>0$,
\begin{equation}\label{e.limp}
  |\p_i\t -\ip_i| \leq C_k \cdot e^{-a_k t}, \quad \forall\, i\in V(H_r).
\end{equation}
This is shown to hold for $r=r_1$ ($k=1$) on account of \qr{e.pconvH} in
the proof of \thmref{t.converge1}. The induction step comprise of
showing that, for $r>0$, we can extend \eqref{e.limp} to hold for
$i\in V(F_r)$ for some constants $a_{k+1}>0$ and $C_{k+1}>0$.

The following two relations, \qr{e.cyext2} and \qr{e.cyext3}, takes us
a good way through the induction step. Let
$\lcy\t=\log\left(\cy\t/\cy\t[0]\right)$.  If $r>0$ then
\begin{equation}
   \lcy\ij\t = \left(r-1\right)t + \Ordo{1}
   \label{e.cyext2}
\end{equation}
for all $ij\in F$ and 
\begin{equation}
   \lcy\ij\t \leq \left(r'-1\right)t + \Ordo{1},
     \label{e.cyext3}
\end{equation}
for $ij\in G\setminus H'$. We defer the proof of these two claims to
the end.

We use \lemref{l.dcont} together with \qr{e.cyext2} and \qr{e.cyext3}
to conclude the proof.  We want to show that
$$\p\t = \mc D\left( \cy\t\vert_{G}, \p\t\vert_{V(H)} \right), $$
satisfies $|\p_i\t - \ip_i| = \Ordo{e^{-a_{k+1}t}}$ for $i\in V(H')$.
Since $\ipsi\vert_{F} \equiv r$ is a constant field, we know that
$F=H'\setminus H$ contains only zero-cost oriented cycles.  Hence, by
\lemref{l.pnear},
\begin{equation}\label{e.pias}
\ip\vert_{V(H')} = 
\mc D\left(\ff'\vert_{H'},\ip\vert_{V(H)}\right),
\end{equation}
where $\ff'\t\in \Phi^+(b)$ is the flow supported on $H'$ obtained by
applying \lemref{l.removal} and removing the arcs $G\setminus H'$ from
the positive flow $\ff\t$ in \qr{e.xphi}. 
If we let
\begin{equation}
\q\t = \mc D\left( \cy\t\vert_{G}, \ip\vert_{V(H)} \right),\label{e.qdef}
\end{equation}
then $|\p_i\t - \q_i\t|\leq \Ordo{e^{-a_{k}t}}$ for $i\in V(H')$, by the
maximum principle and the induction hypothesis \eqref{e.limp}. It is
therefore enough to show that
\begin{equation}
  \label{e.goal}
  |\q_i\t - \ip_i| = \Ordo{e^{-a_{k+1} t}},
  \quad\forall\, i\in V(F). 
\end{equation}

By
\qr{e.cyext3}, \qr{e.xphi} and \lemref{l.removal}, 
\begin{equation}
\normi{\ff'\t - \cy\t} = \Ordo{e^{(r'-1)t}}.\label{e.xxx}
\end{equation}
and by \qr{e.cyext2}, \qr{e.xphi} and \lemref{l.removal}, 
\begin{equation}
\ff'\ij\t = \OrdoTheta{e^{(r-1)t}} +  \Ordo{e^{(r'-1)t}} =
\OrdoTheta{e^{(r-1)t}} ,\label{e.xxx1}
\end{equation}
for $ij\in H'$. It follows that the conditions of \lemref{l.dcont}
are fulfilled for the solutions in \qr{e.qdef} and \qr{e.pias} with
$\cy=\cy\t$ and $\hat\cy=\ff'\t$ and
$$\delta(\cy\t\vert_{G},\ff'\vert_{G}) =
\Ordo{\exp\{-(r-r')t\}}, $$
and thus we can conclude from \lemref{l.dcont} that \qr{e.goal}
holds with 
$$a_{k+1} = \min\{a_k,r-r'\}.$$  
\qed

\begin{proof}[Proof of \eqref{e.cyext2} and \eqref{e.cyext3}]
 Integration gives that if $|p_i\t -
\ip_i|=\Ordo{e^{-a_{k}t}}$
then 
\begin{equation}
\lvert{\bar\p}_i\t-\ip_i\rvert=\Ordo{1/t}.\label{e.intot}
\end{equation}

We first show that
\begin{equation}
  \lcy\ij\t \leq \left(r-1\right)t + \Ordo{1}, 
\qquad\text{for all $ij\in G\setminus H$}.
     \label{e.cyext1}
\end{equation}
Given $kl\in G\setminus H$, we obtain either that $\lcy_{kl}\t\leq -t
+ \Ordo1$, or else, by \qr{e.crj}, a basic solution
$\bfs=\bfs\t\in\BFS$ such that $\bfs_{kl}>0$ and
$$ \min\lcy(\supp\bfs) \geq  \lcy\kl\t + \Ordo1. $$
The forest $\bfs$ contains a unique directed $ab$-path $\pi$ containing
$kl$ with endpoints $a,b$ in $H$ and internal vertices in $G\setminus
H$. Then \qr{e.gmean1} and \qr{e.intot} gives that
\begin{align*}
  \lcy\kl\t &\leq M_\pi(\lcy) + \Ordo1 
  = t(\frac{\bar\p_a-\bar\p_b}{\l(\pi)} - 1) + \Ordo1\\
  &= t(\frac{\ip_a-\ip_b}{\l(\pi)} - 1) + t\Ordo{1/t} + \Ordo1\\
  &\leq t(r - 1) + \Ordo1,
\end{align*}
since, by the construction of $H=H_r$, we know that $\pi$ has
$\ip$-slope at most $r=r_k$. This shows \qr{e.cyext1}. 

Furthermore, it follows by the construction of $F=F_r$ that, if $ij\in
F$, then $ij$ is the endpoint of an $ab$-path $\pi$, $a,b\in V(H)$,
which has $\ip$-slope exactly equal to $r$. Thus the $\bar\p\t$-slope
of $\pi$ equals $r+\Ordo{1/t}$, by \qr{e.intot}. But, by
\qr{e.dicond5} and \qr{e.cyext1},
$$\bar\psi\ij = \lcy\ij\t/t + 1 \leq r + \Ordo{1/t}, $$
and since the $\bar\p$-slope of $\pi$ equals the average 
$$ M_\pi(\bar\psi) = \frac1{\l(\pi)} \sum_{ij\in\pi} \l\ij\bar\psi, $$ 
we obtain
$$ 
 \bar\psi\ij \geq r + \frac{\l(\pi)}{\l\ij} \Ordo{1/t} = r+\Ordo{1/t}, 
\quad \forall\, ij\in \pi.
$$
Together with \qr{e.dicond5} and \qr{e.cyext1} this shows \qr{e.cyext2}. 

It also follows that $\bar\psi\ij\t = r + \Ordo{1/t}$ for $ij\in
F$. Hence,
\begin{equation}
\bar\p_i\t = \ip_i + \Ordo{1/t} \quad\text{for $i\in V(H')$.}\label{e.apprih}
\end{equation}
Hence,  we can use the
same argument as for \qr{e.cyext1} to deduce \qr{e.cyext3}, since that
argument only used \qr{e.intot} which can be replaced by
\qr{e.apprih}. 
\end{proof}

\section*{Acknowledgements}

This work was funded by the Human Frontiers Science Program grant
RGP51/2007.

\bibliographystyle{plain} \bibliography{physarum}

\end{document}